\documentclass[12pt,reqno]{amsart}

\usepackage{amsthm,amssymb,amstext,amscd,amsfonts,amsbsy,amsxtra,latexsym,amsmath,xcolor,mathrsfs,fancybox}
\usepackage[top=1.5in, bottom=2in, left=2.5in, right=2.5in]{geometry}
\usepackage[english]{babel}
\usepackage[latin1]{inputenc}
\usepackage{cancel}
\usepackage{comment,enumitem}
\usepackage{mdframed}
\allowdisplaybreaks
\newcommand{\field}[1]{\mathbb{#1}}
\newcommand{\N}{\field{N}}
\newcommand{\Z}{\field{Z}}
\newcommand{\R}{\field{R}}
\newcommand{\C}{\field{C}}

\renewcommand{\H}{\mathbb{H}}

\newcommand{\SL}{\operatorname{SL}}

\newcommand{\re}{\operatorname{Re}}

\newcommand{\bea}{\begin{eqnarray}}
\newcommand{\eea}{\end{eqnarray}}
\newcommand{\be}{\begin {equation}}
\newcommand{\ee}{\end{equation}}

\newcommand{\bpmod}[1]{\quad\left(\text{mod }#1\right)}
\numberwithin{equation}{section}
\newtheorem{theorem}{Theorem}
\numberwithin{theorem}{section}
\newtheorem{lemma}[theorem]{Lemma}

\newtheorem{corollary}[theorem]{Corollary}
\theoremstyle{remark}

\newtheorem*{remarks}{Remarks}
\theoremstyle{definition}
\newtheorem*{definition}{Definition}

\renewenvironment{proof}[1][Proof]{\begin{trivlist} \item[\hskip \labelsep {\bfseries #1:}]}{\qed\end{trivlist}}

\begin{document}	
\title[Regularized inner products and Hecke eigenforms]{Regularized inner products and weakly holomorphic Hecke eigenforms}
\author{Kathrin Bringmann}
\address{Mathematical Institute\\University of
	Cologne\\ Weyertal 86-90 \\ 50931 Cologne \\Germany}
\email{kbringma@math.uni-koeln.de}
\author{Ben Kane}
\address{Department of Mathematics\\ University of Hong Kong\\ Pokfulam, Hong Kong}
\email{bkane@hku.hk}
\date{\today}
\thanks{The research of the first author is supported by the Alfried Krupp Prize for Young University Teachers of the Krupp foundation and the research leading to these results receives funding from the European Research Council under the European Union's Seventh Framework Programme (FP/2007-2013) / ERC Grant agreement n. 335220 - AQSER.  The research of the second author was supported by grants from the Research Grants Council of the Hong Kong SAR, China (project numbers HKU 27300314, 17302515, and 17316416).}
\maketitle

\section{Introduction and statement of results}\label{sec:intro}
Hecke operators play a central role in the study of modular forms. The classical theory of newforms and modular $L$-functions relies on the fact that spaces of cusp forms are finite-dimensional and can be diagonalized with respect to the Hecke algebra. Although spaces of weakly holomorphic modular forms are not finite-dimensional, it turns out that Hecke operators still play an important role and there is a natural subspace whose elements can be considered as trivial. Then a natural Hecke structure can be described.
	
For $\kappa\in\Z$, denote by $M_{2\kappa}^!$ the space of weight $2\kappa$ weakly holomorphic modular forms (i.e., those meromorphic modular forms whose only possible pole occurs at $i\infty$) with respect to the full modular group $\SL_2(\Z)$.  Motivated by applications coming from the theory of $p$-adic modular forms, Guerzhoy \cite{Guerzhoy} observed that even though the space $M_{2k}^!$ ($k\in \N$) is infinite-dimensional and the Hecke operators $T_m$ (defined in \eqref{eqn:Tpdef}) increase the order of the pole, there is a meaningful Hecke theory on this space.  Emulating the usual definition of Hecke eigenforms (i.e., for every $m\in\N$ there exists $\lambda_m\in\C$ such that $f|T_m=\lambda_m f$), Guerzhoy defined a weakly holomorphic Hecke eigenform to be any $f\in M_{2k}^!\backslash\mathcal{J}$ ($\mathcal{J}$ is a specific subspace of $M_{2k}^!$ which is stable under the Hecke algebra) which satisfies
\[
f|T_{m} \equiv \lambda_m f\pmod{\mathcal{J}}.
\]
Here and throughout, for a vector space $\mathcal{S}$ and a subspace $\mathcal{J}\subset \mathcal{S}$, for $f,g\in \mathcal{S}$ the congruence $f\equiv g\pmod{\mathcal{J}}$ means that $f-g\in \mathcal{J}$.  In \cite{Guerzhoy}, Guerzhoy chose $\mathcal{J}=D^{2k-1}(M_{2-2k}^!)$, where $D:=\frac{1}{2\pi i} \frac{\partial}{\partial z}$. 
To understand Guerzhoy's motivation, recall that Bruinier, Ono, and Rhoades \cite{BOR} showed that $D^{2k-1}$ is an injective map from the  space of so-called harmonic Maass forms (non-holomorphic generalizations of modular forms) into the space $M_{2k}^!$ and determined its image. Guerzhoy realized that the image of the distinguished subspace $M_{2-2k}^!$ was a natural object to factor out by. However, it was later determined that $\mathcal{J}=D^{2k-1}(S_{2-2k}^!)$ is a better choice (cf. \cite{BGKO,Kent}), where $S_{2\kappa}^!$ is the space of \begin{it}weak cusp forms\end{it} of weight $2\kappa\in 2\Z$, i.e., those weakly holomorphic modular forms with vanishing constant term. Note that $D^{2k-1}$ and $T_m$ essentially commute (see \eqref{eqn:DHeckeEquiv}) and $S_{2-2k}^!$ is preserved by $T_m$, so $D^{2k-1}(S_{2-2k}^!)|T_m\subseteq D^{2k-1}\left(S_{2-2k}^!\right)$, and hence $D^{2k-1}(S_{2-2k}^!)$ is stable under the Hecke operators.  Thus a \begin{it}weakly holomorphic Hecke eigenform\end{it} is any $f\in M_{2k}^!\setminus D^{2k-1}(S_{2-2k}^!)$ for which there exist $\lambda_m\in \C$ satisfying (for all $m\in\N$) 
\begin{equation}\label{eqn:fTp}
f|T_m \equiv  \lambda_m f \bpmod{D^{2k-1}\!\left(S_{2-2k}^!\right)}.
\end{equation}
Functions $f$ satisfying \eqref{eqn:fTp} are not eigenfunctions of $T_m$ in the usual sense. However a reasonable interpretation is to consider $[f]:=f+D^{2k-1}(S_{2-2k}^!)$ as an element of the quotient space $M_{2k}^!/D^{2k-1}(S_{2-2k}^!)$. Functions $f$ satisfying (\ref{eqn:fTp}) may then be viewed as eigenvectors $[f]$ of the Hecke operators in that factor space. As is usual, we exclude $[0]=D^{2k-1}(S_{2-2k}^!)$ in the definition of eigenvector.

Viewed in this light, the definition of weakly holomorphic Hecke eigenforms is perhaps not very enlightening at first glance.  If one simply defines ``eigenforms'' as elements of some quotient space, then one can replace $D^{2k-1}(S_{2-2k}^!)$ with any space $\mathcal{J}$ that is preserved under the action of the Hecke operators.  It is hence natural to ask why $D^{2k-1}(S_{2-2k}^!)$ is the correct subspace. There are a few answers to this question. The initial perspective taken in \cite{Guerzhoy} was a $p$-adic one, because for $f\in D^{2k-1}(S_{2-2k}^!)$ with integral coefficients the $(p^mn)$-th coefficients become divisible by higher powers of $p$ as $m$ gets larger. Elements of $D^{2k-1}(S_{2-2k}^!)$ have vanishing period polynomials \cite{BGKO} and there is a recent cohomological interpretation as a space of coboundaries announced by Funke. Moreover certain regularized critical $L$-values are zero for elements in this space (see \cite[Theorem 2.5]{BFK}).

In this paper, we give another reason for the choice $\mathcal{J}=D^{2k-1}(S_{2-2k}^!)$ by viewing these Hecke eigenforms in the framework of a regularized inner product $\langle~, ~\rangle$, defined in \cite{BDE} for arbitrary weakly holomorphic modular forms.
Note that if $f\in M_{2k}$ is a classical holomorphic Hecke eigenform, then for each $m\in\N$ there exists $\lambda_m\in \C$, for which
\begin{equation}\label{eqn:fTpinner}
	\left<f|T_m, g\right> =\lambda_m\!\left<f,g\right>
\end{equation}
for all $g\in M_{2k}$.  Indeed, since the inner product is non-degenerate on $M_{2k}$ (see \cite[Section 5]{ZagierNotRapid}), (\ref{eqn:fTpinner}) is equivalent to the usual definition of Hecke eigenforms on $M_{2k}$.

One may hence treat (1.2) as an alternative definition of Hecke eigenforms. Both from this perspective and from the perspective of investigating the inner product, it is very natural to factor out by the space 
\begin{equation}\label{eqn:orth}
M_{2k}^{!,\perp}:=\left\{f\in M_{2k}^!: \left<f,g\right>=0\ \text{ for all } g\in M_{2k}^!\right\}
\end{equation}
of weakly holomorphic forms which are orthogonal to all of $M_{2k}^!$. 

The main point of this paper is to give another motivation of the definition in \cite{Guerzhoy} by proving that the space $D^{2k-1}(S_{2-2k}^!)$ is the ``degenerate part'' of $M_{2k}^!$ in the sense that it is orthogonal to all of $M_{2k}^!$, from which one also conversely concludes that $M_{2k}^{!,\perp}$ is a very large subspace.
\begin{theorem}\label{thm:degenerate}
A function $f\in M_{2k}^!$ is in $M_{2k}^{!, \perp}$ if and only if $f\in D^{2k-1}(S_{2-2k}^!)$.
\end{theorem}
As alluded to above, Theorem \ref{thm:degenerate} yields another characterization of weakly holomorphic Hecke eigenforms.
\begin{corollary}\label{cor:Heckeequiv}
A function $f\in M_{2k}^!$ is a weakly holomorphic Hecke eigenform if and only if for every $m\in\N$ there exists $\lambda_m$ satisfying \eqref{eqn:fTpinner} for all $g\in M_{2k}^!$.
\end{corollary}
\begin{remarks}
\noindent

\noindent
\begin{enumerate}[leftmargin=*, align=left, label={\rm(\arabic*)}]
\item
As shown in \cite[Theorem 1.3]{BDE}, the inner product is Hermitian on $M_{2k}^!$. Hence if $f$ is a weakly holomorphic Hecke eigenform, then \eqref{eqn:fTpinner} implies that for all $g\in M_{2k}^!$
\[
\left<g,f|T_m\right> = \overline{\lambda_{m}}\left<g,f\right>.
\]
Moreover, by Lemma \ref{lem:TpHermitian} below, the Hecke operators are also Hermitian. Thus, taking $g=f$, we see that $\lambda_m\in \R\text{ or } \left<f,f\right>=0$.
\item 
The regularized inner product plays an important role in both mathematics and physics, due to the fact that Borcherds used this regularization to construct theta lifts of weakly holomorphic modular forms on orthogonal groups. For example, such lifts recently occurred in Angelantonj, Florakis, and Pioline's investigation \cite{PiolineOneLoop} of BPS amplitudes with momentum insertions (see \cite[(3.23)]{PiolineOneLoop}). The same authors evaluate higher-dimensional analogues of these lifts in their continued study of one-loop integrals in \cite[Section 2.2]{PiolineThreshold}.

\rm
\item
The regularized inner product that Borcherds used was based on a regularization by physicists Harvey and Moore \cite{HM} and the regularization considered in this paper agrees with Harvey and Moore's regularization whenever theirs exists. Moreover, weakly holomorphic modular forms and mock modular forms appear in moonshine (such as the $j$-function occurring in McKay--Thompson's Monstrous Moonshine) and the theory of quantum black holes (see \cite{DMZ}). Due to these connections, it may be of interest to physicists to understand the structure of the inner product better; specifically, we discuss next whether the space of such functions has some kind of Hilbert space structure. As seen in Theorem \ref{thm:degenerate}, $M_{2k}^!$ contains a large degenerate subspace $D^{2k-1}(S_{2-2k}^!)$, and thus the space is itself not a Hilbert space. However, combining Theorem \ref{thm:degenerate} with \cite[Theorem 1.2]{BGKO}, there is an isomorphism 
\begin{equation*}
M_{2k}^!\Big/M_{2k}^{!,\perp} = M_{2k}^!\Big/D^{2k-1}\!\left(S_{2-2k}^!\right) \cong M_{2k}\oplus M_{2k}.
\end{equation*}
In particular, a natural subspace is isomorphic to two copies of $S_{2k}$. Since $S_{2k}$ is a finite-dimensional Hilbert space under the classical Petersson inner product, there is a naturally occurring Hilbert space of twice the dimension of $S_{2k}$, and the weakly holomorphic Hecke eigenforms form an orthogonal basis of this space. It may hence be natural to investigate physical properties arising from weakly holomorphic modular forms occurring within this subspace.

\end{enumerate}
\end{remarks}

 The paper is organized as follows. In Section \ref{sec:prelim}, we introduce a family of non-holomorphic modular forms known as harmonic Maass forms which play an important role in the proof of Theorem \ref{thm:degenerate} and then
we explain how one can define regularized inner products.  In Section \ref{sec:weakly}, we prove Theorem \ref{thm:degenerate} and Corollary \ref{cor:Heckeequiv}.
\section{Preliminaries}\label{sec:prelim}
\subsection{Harmonic Maass forms}
We begin by defining harmonic Maass forms, which were first introduced by Bruinier--Funke \cite{BF}.
\begin{definition}
For $\kappa\in\mathbb{Z}$, a {\it harmonic Maass form of weight $2\kappa$} is a function $F:\mathbb{H}\to\mathbb{C}$, which is real analytic on $\H$ and satisfies the following conditions:
\begin{enumerate}[leftmargin=*, align=left, label={\rm(\arabic*)}]
\item For every $M=\left(\begin{smallmatrix} a&b\\c&d\end{smallmatrix}\right)\in\operatorname{SL}_2(\mathbb{Z})$, we have
\[
F\left(\frac{az+b}{cz+d}\right)=(cz+d)^{2\kappa}F(z).
\]
\item We have $\Delta_{2\kappa}(F)=0$, where $\Delta_{2\kappa}$ is the weight $2\kappa$ {\it hyperbolic Laplace operator} $(z=x+iy)$
$$
\Delta_{2\kappa}:=-y^2\left(\frac{\partial^2}{\partial x^2}+\frac{\partial^2}{\partial y^2}\right)+2\kappa iy\left(\frac{\partial}{\partial x}+i\frac{\partial}{\partial y}\right).
$$
\item The function $F$ grows at most linear exponentially at $i\infty$.
\end{enumerate}
%\hspace*{-4.5mm}
\end{definition}
A weight $2\kappa$ harmonic Maass form $F$ has a Fourier expansion of the type ($c_F(n,y)\in\C$)
\begin{equation}\label{eqn:FourierGeneral}
F(z)=\sum_{n\in\Z} c_F(n,y)e^{2\pi i nz}.
\end{equation}
If $F$ is weakly holomorphic, then $c_F(n,y)=c_F(n)\in\C$ is independent of $y$.
 More generally, one may use condition (2) above to determine the dependence on $y$ and conclude that \eqref{eqn:FourierGeneral} naturally decomposes into holomorphic and non-holomorphic parts.  Namely, for a harmonic Maass form $F$ of weight $2-2k<0$,
\begin{equation}\label{eqn:Fourier}
F(z)=F^+(z)+F^-(z)
\end{equation}
where, for some (unique) $c_F^\pm(n)\in\mathbb{C}$, we have
\begin{align*}
F^+(z)&=\sum_{n\gg -\infty}c^+_F(n)e^{2\pi i nz},\\
F^-(z)&=c^-_F(0)y^{2k-1}+\sum_{\substack{n\ll\infty \\ n\neq 0}}c^-_F(n)\Gamma(2k-1, -4\pi ny)e^{2\pi i nz}.
\end{align*}
Here the \begin{it}incomplete gamma function\end{it} is given by $\Gamma(\alpha, w):=\int_{w}^{\infty}e^{-t}t^{\alpha-1}dt$ (for $\re(\alpha)>0$ and $w\in\C$).  We call $F^+$ its {\it holomorphic part} of $F$ and $F^-$ the {\it non-holomorphic part}.

We let $H_{2\kappa}^{\operatorname{mg}}$ be the space of weight $2\kappa$ harmonic Maass forms.   The operators $\xi_{2-2k}:=2iy^{2-2k} \overline{\frac{\partial}{\partial \overline{z}}}$ and $D^{2k-1}$, defined in the introduction, map $H_{2-2k}^{\operatorname{mg}}$ to $M_{2k}^!$, and we let $H_{2-2k}$ be the subspace of $H_{2-2k}^{\operatorname{mg}}$ consisting of forms which map to cusp forms under $\xi_{2-2k}$.  These differential operators also act naturally on the Fourier expansion \eqref{eqn:Fourier}.  In particular, as collected in \cite[Theorems 5.5 and 5.9]{Book}, 
 \begin{align*}
\xi_{2-2k}(F(z))&= (2k-1)\overline{c_F^-(0)}-(4\pi)^{2k-1}\sum_{\substack{n\gg -\infty\\ n\neq 0}} n^{2k-1}\overline{c_{F}^-(-n)} e^{2\pi i nz},\\
D^{2k-1}(F(z))&=-\frac{(2k-1)!}{(4\pi)^{2k-1}}c_F^-(0)+ \sum_{n\gg -\infty} n^{2k-1} c_F^+(n) e^{2\pi i nz}.
\end{align*}

Recall that the operators $D^{2k-1}$ and $\xi_{2-2k}$ are Hecke-equivariant (cf. \cite[(7.4) and (7.5)]{Book}), i.e., for any harmonic Maass form $F$
\begin{align}
\label{eqn:xiHeckeEquiv}
\xi_{2-2k}\left(F|_{
2-2k
}T_m\right) &=m^{1-2k} \xi_{2-2k}(F)|_{
2k
}T_m,\\
\label{eqn:DHeckeEquiv}
D^{2k-1}\left(F|_{2-2k}T_m\right) &=m^{1-2k}D^{2k-1}(F)|_{2k}T_m.
\end{align}
Here for $\kappa\in\Z$, the Hecke operators $T_m:H_{2\kappa}^{\textnormal{mg}}\rightarrow H_{2\kappa}^{\textnormal{mg}}$ are given on the expansion \eqref{eqn:FourierGeneral} by
\begin{equation}\label{eqn:Tpdef}
F(z)|_{2\kappa}T_m:=\sum_{n\in\Z} \sum_{d\mid (m,n)} d^{2\kappa-1} c_F\!\left(\frac{mn}{d^2},\frac{d^2 y}{m}\right)e^{2\pi i nz}.
\end{equation}

\subsection{Regularized inner products}\label{sec:reginner}
For two cusp forms $f,g\in S_{2k}$, Petersson's classical inner product is defined by
\begin{equation}\label{classical}
\left<f,g\right>:=\int_{\SL_2(\Z)\backslash\H} f(z)\overline{g(z)} y^{2k} \frac{dxdy}{y^2}.
\end{equation}
This was extended to an inner product over all of $M_{2k}^!$ in a series of steps.  The first such attempt to do so appears to be by Petersson himself \cite{Pe2}, which was later rediscovered and extended by Harvey--Moore \cite{HM} and Borcherds \cite{Borcherds} as we describe below.  Setting
\[
\mathcal{F}_{T}:=\left\{ z\in \H: |z|\geq 1,\ y\leq T,\ -\frac{1}{2}\leq x\leq \frac{1}{2}\right\},
\]
the first regularized inner product is defined by
\begin{equation}\label{eqn:innerreg1}
\left<f,g\right>:=\lim_{T\to\infty}\int_{\mathcal{F}_T} f(z)\overline{g(z)} y^{2k} \frac{dx dy}{y^2}.
\end{equation}
Borcherds extended this regularization by multiplying the integrand by $y^{-s}$ and taking the constant term of the Laurent expansion around $s=0$ of the analytic continuation in $s$. The regularization \eqref{eqn:innerreg1} coincides with \eqref{classical} whenever \eqref{classical} converges. Unfortunately, the regularization \eqref{eqn:innerreg1} and even Borcherds's extension also do not always exist.  In particular, one cannot use them to define a meaningful norm for weakly holomorphic forms.
 This problem was overcome to obtain a regularized inner product for arbitrary $f,g\in M_{2k}^!$
by Diamantis, Ehlen, and the first author \cite{BDE}.  The idea is simple. One multiplies the integrands with a function that forces convergence and then analytically continues.

For this, observe that for $\operatorname{Re}(w)\gg0$, the integral
	$$
	I(f,g;w,s):=\int_{\mathcal{F}} f(z)\overline{g(z)} y^{2k-s} e^{-wy} \frac{dxdy}{y^2}
	$$
converges and is analytic, 
while formally plugging in $w=0=s$ yields Petersson's classical inner product \eqref{classical}. 
For every $\varphi\in(\pi/2,3\pi/2)\setminus\{\pi\}$ it has an analytic continuation $I_\varphi(f,g;w,s)$ to $U_\varphi\times\mathbb{C}$ with $U_\varphi\subset\mathbb{C}$ a certain open set 
containing $0$.
 Then define
	$$
	\langle f,g \rangle_\varphi :=\operatorname{CT}_{s=0} I_\varphi(f,g;0,s) -i\sum_{n>0} c_f(-n)\overline{c_g(-n)}\operatorname{Im}\left(E_{2-2k,\varphi}(-4\pi n)\right),
	$$
where $\operatorname{CT}_{s=0}F(s)$ denotes the constant term of an analytic function $F$, $c_F$ are defined as in \eqref{eqn:FourierGeneral}, and $E_{r, \varphi}$ is the generalized exponential integral (see \cite[(8.19.3)]{NIST}) defined with branch cut on the ray $\{xe^{i\varphi}: x\in\mathbb{R}^+\}$. In \cite{BDE} it is shown that for $f,g\in M_{2k}^!$, $\langle f,g\rangle$ exists and is independent of the choice of $\varphi$. Furthermore, it equals the regularization \eqref{eqn:innerreg1} whenever \eqref{eqn:innerreg1} exists.

The following lemma describes that the inner product from \cite{BDE} commutes with the Hecke algebra.
\begin{lemma}\label{lem:TpHermitian}
For $k\in\N$, the Hecke operators are Hermitian on $M_{2k}^!$.
\end{lemma}
\begin{proof}
We split the proof into three cases.  First we assume that $f\in S_{2k}^!$ and $g\in M_{2k}^!$, then we use the fact that the inner product is Hermitian (see \cite[Theorem 1.3]{BDE}) for the case that $f\in M_{2k}^!$ and $g\in S_{2k}^!$, and finally we consider $f=g=E_{2k}$; these cover all cases by linearity.

Following \cite[(1.15)]{BGKO}, one defines, for $f,g\in M_{2k}^!$,
\[
\{f,g\}_0:=\sum_{n\in\Z\backslash \{0\}} \frac{c_{f}(-n)c_g(n)}{n^{2k-1}}.
\]
The ratio $c_{g}(n)/n^{2k-1}$ can also be interpreted as the $n$th coefficient of the holomorphic part of any harmonic Maass form $G$ for which $D^{2k-1}(G)=g$. This yields a connection with the {\it Bruinier--Funke pairing} (cf. \cite{BF} for the original definition restricted to certain subspaces)
\begin{equation*}
%\label{eqn:BFpairing}
\{f,G\}:=\sum_{n\in\Z} c_f(-n)c_{G}^+(n)
\end{equation*}
between $f\in M_{2k}^!$ and $G\in H_{2-2k}^{\operatorname{mg}}$.

We first assume that $f\in S_{2k}^!$ and $g\in M_{2k}^!$.  In this case
\begin{equation}\label{eqn:relateBF}
\{f,g\}_0=\{f,G\}.
\end{equation}
This is related to the inner product by \cite[Theorem 4.1]{BDE}; namely,

\begin{equation}\label{eqn:innerBF}
\{f,G\} = \left<f,\xi_{2-2k}(G)\right>.
\end{equation}
Setting  $h:=\xi_{2-2k}(G)$, we combine \eqref{eqn:innerBF} with \eqref{eqn:relateBF} and use the fact that the Hecke operators are Hermitian with respect to $\left<\cdot,\cdot\right>_0$ to obtain
%(cf. the statement before \cite[Theorem 1.6]{BGKO})
\[
\left<f|_{2k}T_m,h\right> = \left\{f|_{2k} T_m,G\right\} = \left\{ f|_{2k}T_m,g\right\}_0 = \left\{f,g|_{2k}T_m\right\}_0.
\]
Next, by \eqref{eqn:DHeckeEquiv}, we have
\[
D^{2k-1}\!\left(G|_{2-2k}T_m\right) = m^{1-2k}D^{2k-1}(G)|_{2k}T_m =m^{1-2k}g|_{2k}T_m.
\]
Plugging this into \eqref{eqn:relateBF} and using \eqref{eqn:innerBF}, yields
\[
\left\{f,g|_{2k}T_m\right\}_0=m^{2k-1}\left\{ f,G|_{2-2k}T_m\right\} =m^{2k-1}\left<f, \xi_{2-2k}\left(G|_{2-2k}T_m\right)\right>.
\]
We finally use \eqref{eqn:xiHeckeEquiv} to obtain
\[
m^{2k-1}\left<f, \xi_{2-2k}\left(G|_{2-2k}T_m\right)\right>=\left<f, h|_{2k}T_m\right>.
\]
Altogether, we have therefore shown that
\begin{equation}\label{eqn:TpHermitian}
\left<f|_{2k}T_m,h\right>=\left<f,h|_{2k}T_m\right>.
\end{equation}
Since $\xi_{2-2k}$ is surjective on $M_{2k}^!$, we obtain the claim for arbitrary $f\in S_{2k}^!$ and $h\in M_{2k}^!$.

Next suppose that $f\in M_{2k}^!$ and $g\in S_{2k}^!$.  Since the inner product is Hermitian (see \cite[Theorem 1.3]{BDE}), by \eqref{eqn:TpHermitian} we have
\[
\left<f|_{2k}T_m,g\right>=\overline{\left<g,f|_{2k}T_m\right>} = \overline{\left<g|_{2k}T_m,f\right>} = \left<f,g|_{2k}T_m\right>.
\]
Finally, for $f=E_{2k}=g$, we obtain the result directly from the fact that both $f$ and $g$ are Hecke eigenforms with the same real eigenvalues.

\end{proof}

\section{Degeneracy and Proofs of Theorem \ref{thm:degenerate} and Corollary \ref{cor:Heckeequiv}}\label{sec:weakly}
In this section, we show that the space $M_{2k}^{!,\perp}$ defined in \eqref{eqn:orth} coincides with the space $D^{2k-1}(S_{2-2k}^!)$ which appears in the definition of Hecke eigenforms for weakly holomorphic modular forms.    In \cite[Corollary 4.5]{BDE}, the space $M_{2k}^{!,\perp}$, defined in \eqref{eqn:orth}, was explicitly determined.  Theorem \ref{thm:degenerate} rewrites this in a form which is useful for Hecke eigenforms.
In order to prove Theorem \ref{thm:degenerate}, we also require the following useful lemma about the \begin{it}flipping operator\end{it}
\begin{equation*}
%\label{eqn:flipdef}
\mathfrak{F}_\kappa\left(F(z)\right):=-\frac{y^{-\kappa}}{(-\kappa)!}\overline{R_{\kappa}^{-\kappa}\left(F(z)\right)},
\end{equation*}
where 
\[
R_{\kappa}^{\ell}:=R_{\kappa+2\ell-2}\circ \cdots\circ R_{\kappa}
\]
is the \begin{it}repeated raising operator\end{it} with the \begin{it}raising operator\end{it} defined by 
\[
R_{\kappa}:=2i\frac{\partial}{\partial z} + \frac{\kappa}{y}.
\]
The following lemma follows from the calculation in the proof of \cite[Theorem 1.1]{BOR} and \cite[Remark 7]{BOR}; it may be found in this form in \cite[Proposition 5.14]{Book}.
\begin{lemma}\label{lem:flip}
If $F\in H_{2-2k}^{\operatorname{mg}}$, then the operator $\mathfrak{F}_{2-2k}$ satisfies
\begin{align}\label{eqn:flipxi}
\xi_{2-2k}(\mathfrak{F}_{2-2k}(F))&=\frac{(4\pi)^{2k-1}}{(2k-2)!} D^{2k-1}(F),\\
\label{eqn:flipD}
D^{2k-1}(\mathfrak{F}_{2-2k}(F))&=\frac{(2k-2)!}{(4\pi)^{2k-1}}\xi_{2-2k}(F).
\end{align}

Furthermore, $\mathfrak{F}_{2-2k}\circ \mathfrak{F}_{2-2k}(F)=F$.
\end{lemma}

We now use Lemma \ref{lem:flip} to prove Theorem \ref{thm:degenerate}.
\begin{proof}[Proof of Theorem \ref{thm:degenerate}]

By \cite[Corollary 4.5]{BDE},  $f\in M_{2k}^!$ is orthogonal to $M_{2k}^!$ if and only if there exists $F\in H_{2-2k}^{\operatorname{mg}}
$ for which $F^+=0$ and $\xi_{2-2k}(F)=f$.  By \eqref{eqn:flipD}, we have
\begin{equation}\label{eqn:fDflip}
f=\xi_{2-2k}(F)=\frac{(4\pi)^{2k-1}}{(2k-2)!}D^{2k-1}\!\left(\mathfrak{F}_{2-2k}(F)\right).
\end{equation}
Note that $F^+=0$ is equivalent to $D^{2k-1}(F)=0$ and $c_{F}^+(0)=0$.  Since $D^{2k-1}(F)=0$, \eqref{eqn:flipxi} implies that
\[
\xi_{2-2k}\!\left(\mathfrak{F}_{2-2k}(F)\right)
=\frac{(4\pi)^{2k-1}}{(2k-2)!}
 D^{2k-1}(F) = 0.
\]
We conclude that $\mathfrak{F}_{2-2k}(F)\in M_{2-2k}^!$ because this is the kernel of $\xi_{2-2k}$.  Thus by \eqref{eqn:fDflip}, we have $f\in D^{2k-1}(M_{2-2k}^!)$.  Moreover, \cite[displayed formula before Lemma 3.1]{BKR} states that
(note that in this paper the flipping operator is renormalized)
\begin{equation*}
%\label{eqn:constflip}
c_{\mathfrak{F}_{2-2k}(F)}^+(0)=-\overline{c_{F}^+(0)}=0.
\end{equation*}
Thus $F^+=0$ is equivalent to $f\in D^{2k-1}(S_{2-2k}^!)$, which is the claim.
\end{proof}

We conclude the paper by using Theorem \ref{thm:degenerate} to show that one may obtain an alternative characterization of weakly holomorphic Hecke eigenforms by requiring \eqref{eqn:fTpinner} to hold for all $g\in M_{2k}^!$ instead of \eqref{eqn:fTp}.
\begin{proof}[Proof of Corollary \ref{cor:Heckeequiv}]
By Theorem \ref{thm:degenerate}, \eqref{eqn:fTp} is equivalent to
\begin{equation*}
%\label{eqn:fTpequiv}
f|T_m\equiv \lambda_m f\pmod{M_{2k}^{!,\perp}}.
\end{equation*}
Hence $f$ is a weakly holomorphic Hecke eigenform if and only if for every $m$ there exists $\lambda_m\in \C$ such that $h:=f|T_m-\lambda_mf\in M_{2k}^{!,\perp}$. Linearity of the regularized inner product implies that for every $g\in M_{2k}^!$, 
\[
\left<f|T_m,g\right> = \left<\lambda_m f + h,g\right>=\lambda_m\left< f,g\right>+\left<h,g\right>, 
\]
which satisfies  \eqref{eqn:fTpinner} for every $g\in M_{2k}^!$ if and only if $h\in M_{2k}^{!,\perp}$.
\end{proof}


\begin{thebibliography}{99}
%\bibitem{AS} M. Abramovitz and I. Stegun, \begin{it}Handbook of Mathematical functions with formulas, graphs, and mathematical tables\end{it}, 9th edition, New York, Dover, 1972.
\bibitem{PiolineOneLoop} C. Angelantonj, I. Florakis, and B. Pioline, \begin{it}One loop BPS amplitudes as BPS-state sums\end{it}, J. High Energy Phys. \textbf{06:70} (2012), 1--39.
\bibitem{PiolineThreshold} C. Angelantonj, I. Florakis, and B. Pioline, \begin{it}Threshold corrections, generalized prepotentials and Eichler integrals\end{it}, Nucl. Phys. B \textbf{897} (2015), 781--820.


%\bibitem{Beardon} A. Beardon, \begin{it}The geometry of discrete groups\end{it}, Grad. Texts in Math. \textbf{91}, Springer, New York, 1995.
\bibitem{Borcherds} R. Borcherds, \begin{it}Automorphic forms with singularities on Grassmannians,\end{it} Invent. Math. \begin{bf}132\end{bf} (1998), 491--562.
\bibitem{BDE} K. Bringmann, N. Diamantis, and S. Ehlen, \begin{it}Regularized inner products and errors of modularity\end{it},  Int. Math. Res. Not., accepted for publication.
\bibitem{Book}K. Bringmann, A. Folsom, K. Ono, and L. Rolen, \begin{it}Harmonic Maass forms and mock modular forms:  theory and applications\end{it}, in preparation.
\bibitem{BFK} K. Bringmann, K. Fricke, and Z. Kent, \begin{it}Special $L$-values and periods of weakly holomorphic modular forms\end{it}, Proc. Amer. Math. Soc. \textbf{142} (2014), 3425--3439.
\bibitem{BGKO} K. Bringmann, P. Guerzhoy, Z. Kent, and K. Ono, \begin{it}Eichler--Shimura theory for mock modular forms\end{it}, Math. Ann. \textbf{355} (2013), 1085--1121.
%\bibitem{BGK} K. Bringmann, P. Guerzhoy, and B. Kane, \begin{it}Shintani lifts and fractional derivatives for harmonic Maass forms \end{it}, Adv. Math. \textbf{255} (2014), 641--671.
%\bibitem{BJK} K. Bringmann, P. Jenkins, and B. Kane, \begin{it} Polar harmonic Maass form\end{it}, in preparation.
%\bibitem{BKInner} K. Bringmann and B. Kane, \begin{it}Inner products, classical case\end{it}, temporary notes connected to \cite{BJK}.
%\bibitem{BKRamanujan} K. Bringmann and B. Kane, \begin{it}Ramanujan and coefficients of meromorphic modular forms\end{it}, submitted for publication.

%\bibitem{BKlift} K. Bringmann and B. Kane, \begin{it}Polar harmonic lifts\end{it}, in preparation.

%\bibitem{BKFourier} K. Bringmann and B. Kane, \begin{it}Fourier coefficients of meromorphic modular forms and a conundrum of Petersson\end{it}, preprint.

%\bibitem{BKweight0}  K. Bringmann and B. Kane, \begin{it}An explicit weight 0 Riemann--Roch Theorem\end{it}, preprint.
%\bibitem{BKW} K. Bringmann, B. Kane, and W. Kohnen, \begin{it}Locally harmonic Maass forms and the kernel of the Shintani lift\end{it}, Int. Math. Res. Not. \textbf{2015} (2015), 3185--3224.

\bibitem{BKR} K. Bringmann, B. Kane, and R. Rhoades, \begin{it}Duality and differential operators for harmonic Maass forms\end{it}, Dev. Math. \textbf{28} (2013), 85--106.


%\bibitem{BK} K. Bringmann and B. Kane, \begin{it}Fourier coefficients of meromorphic modular forms and a question of Petersson\end{it}, preprint.





%\bibitem{BO} K. Bringmann and K. Ono, \begin{it}Coefficients of harmonic Maass forms\end{it} in Proceedings of the 2008 University of Florida Conference on Partitions, $q$-series and modular forms, Dev. \textbf{23} (2012), 23--38.

\bibitem{BF} J. Bruinier and J. Funke, \begin{it}On two geometric theta lifts\end{it}, Duke Math J. \textbf{125} (2004), 45--90.

\bibitem{BOR} J. Bruinier, K. Ono, and R. Rhoades, \begin{it}Differential operators for harmonic weak Maass forms and the vanishing of Hecke eigenvalues\end{it}, Math. Ann. \textbf{342} (2008), 673--693.

\bibitem{DMZ} A. Dabholkar, S. Murthy, and D. Zagier, \begin{it}Quantum black holes, wall crossing, and mock modular forms\end{it}, Cambridge monographs in mathematical Physics, to appear.

%\bibitem{DITRamanujan} W. Duke, \"O. Imamo$\overline{\text{g}}$lu, and \'A. T\'oth, \begin{it}Regularized inner products of modular functions\end{it}, Ramanujan J., to appear.

%\bibitem{DJ} W. Duke and P. Jenkins, \begin{it}On the zeros and coefficients of certain weakly holomorphic modular forms\end{it}, Pure Appl. Math Q. \textbf{4} (2008), 1327--1340.

%\bibitem{Fa} J. Fay, {\it{Fourier coefficients of the resolvent for a Fuchsian group}}, J. reine angew. Math. \textbf{293-294} (1977), 143--203.

%\bibitem{Folland} G. Folland, Real analysis, Wiley, 1999.

%\bibitem{GS} D. Goldfeld and P. Sarnak, \begin{it}Sums of Kloosterman sums\end{it}, Invent. Math. \textbf{71} (1983), 243--250.


\bibitem{Guerzhoy}P. Guerzhoy, \begin{it}Hecke operators for weakly holomorphic modular forms and supersingular congruences\end{it}, Proc. Amer. Math. Soc. \textbf{136} (2008), 3051--3059.

%\bibitem{GKO} P. Guerzhoy, Z. Kent, and K. Ono, \begin{it}$p$-adic coupling of mock modular forms and shadows\end{it}, Proc. Natl. Acad. Sci. (USA) \textbf{107} (2010), 6169--6174.

%\bibitem{HR3} G. Hardy and S. Ramanujan, \begin{it}On the coefficients in the expansions of certain modular functions\end{it}. Proc. Royal Soc. A \textbf{95} (1918), 144--155.

\bibitem{HM} J. Harvey and G. Moore, {\it Algebras, BPS states, and strings}, Nuclear Phys. B \textbf{463} (1996), 315--368.


%\bibitem{Hejhal} D. Hejhal, The Selberg trace formula for $\operatorname{PSL}_2(\R)$, Lecture Notes in Mathematics \textbf{1001}, 1983.

%\bibitem{ImOs} \"O. Imamo$\overline{\text{g}}$lu and C. O'Sullivan, \begin{it}Parabolic, hyperbolic and elliptic Poincar\'e series\end{it}, Acta Arith. \textbf{139} (2009), 199--228.

%\bibitem{IwaniecKowalski} H. Iwaniec and E. Kowalski, \begin{it}Analytic number theory\end{it}, Providence, RI, American Mathematical Society, 2004.

\bibitem{Kent}Z. Kent, \begin{it}Periods of Eisenstein series and some applications\end{it}, Proc. Amer. Math. Soc. \textbf{139} (2011), 3789--3794.

%\bibitem{Knopp1}M. Knopp, \begin{it} On the construction of certain automorphic forms of non-negative dimension\end{it}, Thesis, Univ. Illinois at Urbana--Champaign, 1958.

%\bibitem{Knopp2} M. Knopp, \begin{it} Automorphic forms of nonnegative dimension and exponential sums\end{it}, Michigan Math. J. \textbf{7} (1960), 257--287.

%\bibitem{Knopp3}M. Knopp, \begin{it}On Abelian integrals of the second kind and modular functions\end{it}, Amer. J. Math. \textbf{84} (1962), 615--628.

%\bibitem{Knopp4} M. Knopp, \begin{it} On generalized abelian integrals of the second kind and modular forms of dimension zero\end{it}, Amer. J. Math. \textbf{86} (1964), 430--440.

%\bibitem{Kohnennewform} W. Kohnen, \begin{it}Newforms of half-integral weight\end{it}, J. reine angew. Math. \textbf{333} (1982), 32--72.

%\bibitem{LagLi} J. Lagarias and W.-C. Li, \begin{it}The Lerch zeta function IV:  Hecke operators\end{it}, preprint.

%\bibitem{Niebur} D. Niebur, \begin{it}A class of nonanalytic automorphic functions\end{it}, Nagoya Math. J. \textbf{52} (1973), 133--145.

%\bibitem{N2} D. Niebur, \begin{it}Construction of automorphic forms and integrals\end{it}, Trans. Amer. Math. Soc. \textbf{191} (1974), 373--385.

\bibitem{NIST} \begin{it}NIST Digital Library of Mathematical Functions\end{it}, {http://dlmf.nist.gov/'}, Release 1.0.10 of 2015-08-07.  Online companion to \cite{Olver}.

\bibitem{Olver}F. Olver, D. Lozier, R. Boisvert, and C. Clark, \begin{it}NIST handbook of mathematical functions\end{it}, Cambridge University Press, New York, NY, 2010, print companion to \cite{NIST}.
	
%\bibitem{Parson}A. Parson, \begin{it}Generalized Kloosterman sums and the Fourier coefficients of cusp forms\end{it}, Trans. Amer. Math. Soc. \textbf{217} (1976), 329--350.

%\bibitem{Pe4} H. Petersson, \begin{it}\"Uber eine Metrisierung der automorphen Formen und die Theorie der Poincar\'eschen Reihen\end{it}, Math. Ann. \textbf{117} (1940), 453--537.

%\bibitem{PeEinheit} H. Petersson, \begin{it}Einheitliche Begr\"undung der Vollst\"andigkeitss\"atze f\"ur die Poincar\'eschen Reihen von reeller Dimension bei beliebigen Grenzkreisgruppen von erster Art\end{it} Abh. Math. Sem. Univ. Hmbg. \textbf{14} (1941), 22--60.

%\bibitem{Pe} H. Petersson, \begin{it}Ein Summationsverfahren f\"ur die Poincar\'eschen Reihen von der Dimension 2 zu den hyperbolischen Fixpunkten\end{it}, Math. Z. \textbf{49} (1943), 441--496.



%\bibitem{Pe3} H. Petersson, \begin{it}Zur analytische Theorie der Grenzkreisgruppen V: Theorie der Poincar\'eschen Reihen zu den hyperbolischen Fixpunktepaaren bei beliebigen Grenzkreisgruppen\end{it}, Math. Z. \textbf{44} (1938), 127--155.

%\bibitem{PeBereich} H. Petersson, \begin{it}\"Uber den Bereich absoluter Konvergenz der Poincar\'eschen Reihen\end{it}, Acta Math. \textbf{80} (1948), 23--63.

%\bibitem{Pe6} H. Petersson, \begin{it}Automorphe Formen als metrische Invarianten, Teil I\end{it}, Math. Nachr. \textbf{1} (1948), 158--212.

%\bibitem{Pe1} H. Petersson, \begin{it}Konstruktion der Modulformen und der zu gewissen Grenzkreisgruppen geh\"origen automorphen Formen von positiver reeller Dimension und die vollst\"andige Bestimmung ihrer Fourierkoeffzienten\end{it}, S.-B. Heidelberger Akad. Wiss. Math. Nat. Kl. (1950), 415--474.

\bibitem{Pe2} H. Petersson, \begin{it}\"Uber automorphe Orthogonalfunktionen und die Konstruktion der automorphen Formen von positiver reeller Dimension\end{it}, Math. Ann. \textbf{127} (1954), 33--81.

%\bibitem{Pe5} H. Petersson, \begin{it}Explizite Konstruktion der automorphen Orthogonalfunktionen in den multiplikativen Differentialklassen\end{it}, Math. Nachr. \textbf{16} (1957), 343--368.
%\bibitem{Pe13} H. Petersson, \begin{it}The properties of the representation of the abelian differentials by Poincar\'e's series\end{it} in ``Contributions to Number Theory'', Internat. Colloq. Function Theory (1960), 185--201.


%\bibitem{Rudin} W. Rudin, Principles of Mathematical Analysis, McGraw--Hill, 1976.

%\bibitem{Siegel}C. Siegel, \begin{it}Berechnung von Zetafunktionen an ganzzahligen Stellen\end{it}, Nachr. Akad. Wiss. G\"ottingen Math.-Phys. Kl. II \textbf{1969} (1969), 87--102.

%\bibitem{VP}  A.-M. von Pippich, \begin{it}The arithmetic of elliptic Eisenstein series\end{it}, Ph. D. thesis, 2010.

%\bibitem{Za} D. Zagier \begin{it} The 123 of modular forms \end{it}

\bibitem{ZagierNotRapid}D. Zagier, \begin{it} The Rankin--Selberg method for automorphic functions which are not of rapid decay\end{it}, J. Fac. Sci Tokyo \textbf{28} (1982), 415--438.
\end{thebibliography}
\end{document}